\documentclass[12pt]{article}

\usepackage[margin=.75in]{geometry}
\usepackage{amsfonts}
\usepackage{amssymb}
\usepackage{amsmath}
\usepackage{amsthm}
\usepackage{latexsym}
\usepackage{ifthen}
\usepackage{color}

\newcommand{\ds}{\displaystyle}

\newtheorem{defi}{Definition}[section]

\newtheorem{thm}[defi]{Theorem}
\newtheorem{cor}[defi]{Corollary}
\newtheorem{ex}[defi]{Example}
\newtheorem{remar}[defi]{Remark}

\renewenvironment{proof}{\noindent {\bf Proof:}}{\hfill{$\blacksquare$} \newline}

\makeatletter
\@addtoreset{equation}{section}
\@addtoreset{figure}{section}
\@addtoreset{table}{section}
\makeatother

\newcommand{\mi}{\hspace{2.5em}}
\newcommand{\mhi}{\hspace{0.8em}}

\newcommand{\be}[1]{\begin{equation} \label{#1}}
\newcommand{\ee}{\end{equation}}
\newcommand{\ba}[1]{\begin{array}{#1}}
\newcommand{\ea}{\end{array}}
\newcommand{\bea}[1]{\begin{eqnarray} \label{#1}}
\newcommand{\eea}{\end{eqnarray}}

\newcommand{\bi}{\begin{itemize}}
\newcommand{\ei}{\end{itemize}}
\newcommand{\bd}{\begin{description}}
\newcommand{\ed}{\end{description}}

\newcommand{\ben}{\begin{enumerate}}
\newcommand{\een}{\end{enumerate}}

\newcommand{\bco}[1]{\begin{corollary} \label{#1}}
\newcommand{\eco}{\end{corollary}}

\usepackage{tikz}
\usetikzlibrary{patterns,decorations.pathmorphing}

\usepackage{url}

\title{A Test Matrix for an Inverse Eigenvalue Problem}
\author{G. M. L. Gladwell$^1$, T. H. Jones$^2$ and N. B. Willms$^2$}

\date{\today}

\renewcommand{\vec}[1]{{\bf #1}}

\begin{document}
	\maketitle
\mbox{\small \em $^1$Distinguished Professor Emeritus, Department of Civil and Environmental Engineering,} \linebreak
\mbox{\small \em \indent \ University of Waterloo, Ontario, N2L 3G1} \linebreak
\mbox{\small \em \indent $^2$Department of Mathematics, Bishop's University, Sherbrooke, Quebec, J1M 2H2} 

\hspace{2cm} \newline
\mbox{\small \indent Correspondence should be addressed to N. B. Willms; bwillms@ubishops.ca}

\hspace{2cm} \newline \indent
\mbox{\small Copyright \copyright 2014 G.\ M.\ L.\ Gladwell et.\ al.\ This is an open access article distributed under the} \linebreak
\mbox{\small \indent   Creative Commons Attribution License, which permits unrestricted use, distribution, and repro-} \linebreak
\mbox{\small \indent duction in any medium, provided the original work is properly cited.}
\begin{abstract}
		We present a real symmetric tri-diagonal matrix of order $n$ whose eigenvalues are $\{ 2k \}_{k=0}^{n-1}$ which also satisfies the additional condition that its leading principle submatrix has a uniformly interlaced spectrum, $\{ 2l + 1 \}_{l=0}^{n-2}$.  The matrix entries are explicit functions of the size $n$, and so the matrix can be used as a test matrix for eigenproblems, both forward and inverse.  An explicit solution of a spring-mass inverse problem incorporating the test matrix is provided.
	\end{abstract}
	
\section[] {Introduction}

We are motivated by the following inverse eigenvalue problem first studied by Hochstadt in 1967 \cite{hoch}. Given two strictly interlaced sequences of real values,
  $$ (\lambda_i )_1^n \mhi \mbox{and } \mhi (\lambda_i^o )_1^{n-1} ,$$
with 
 \be{1} 
\lambda_1 < \lambda_1^o < \lambda_2 < \lambda_2^o < \ \cdots \  < \lambda_{n-1} < \lambda_{n-1}^o < \lambda_n , 
 \ee
find the $n\times n$, real, symmetric, tridiagonal matrix, $B$, such that $\lambda(B) = (\lambda_i )_1^n$ are the eigenvalues of $B$, while $\lambda(B^o) = (\lambda_i^o )_1^{n-1}$ are the eigenvalues of the leading principal submatrix of $B$, that is, where $B^o$ is obtained from $B$ by deleting the last row and column. The condition on the data-set \eqref{1} is both necessary and sufficient for the existence of a unique Jacobian matrix solution to the problem (see \cite{ipv} \S 4.3  or \cite{nbw} \S 1.2 for a history of the problem, and \S 3 of this paper for additional background theory).

A number of different constructive procedures to produce the exact solution of this inverse problem have been developed \cite{bk, bg, bw, gb, hald, hoch2}, but none provide an explicit characterization of the entries of the solution matrix, $B$, in terms of the data-set \eqref{1}. Computer implementation of these procedures introduces floating point error and associated numerical stability issues. Loss of significant figures due to accumulation of round-off error makes some of the known solution procedures undesirable. Determining the extent of round-off in the numerical solution, $\hat{B}$, computed from a given data-set requires {\it a priori}\/ knowledge of the exact solution $B$. In the absence of this knowledge, an additional numerical computation of the forward problem to find the spectra, $\lambda ( \hat{B} )$ and $\lambda ( \hat{B}^o)$ allows comparison to the original data.

Test matrices, with known entries and known spectra, are therefore helpful in comparing the efficacy of the various solution algorithms in regards to stability. It is particularly helpful when said test matrices can be produced at arbitrary size. However some existant test matrices given as a function of matrix size $n$ suffer the following trait: when ordered by size, the minimum spacing between consecutive eigenvalues is a decreasing function of $n$. This trait is potentially undesirable since the reciprocal of this minimum separation between eigenvalues can be thought of as a condition number on the sensitivity of the eigenvectors (invariant subspaces) to perturbation (see \cite{mc}, Theorem 8.1.12). Some of the algorithms for the inverse problem seem to suffer from this form of ill-conditioning. From a motivation to avoid confounding the numerical stability issue with potential increased ill-conditioning of the data-set as a function of $n$, the authors developed a test matrix which has equally spaced, and uniformly interlaced, simple eigenvalues.

In Section \S 2 we provide the explicit entries of such a matrix, $A(n)$. We claim that its eigenvalues are equally spaced as 
$$ \lambda (A(n) ) = \{ 0,\ 2,\ 4,\ \ldots,\ 2n-2 \},$$
while its leading principal submatrix ${\ds A^o(n)}$ has eigenvalues uniformly interlaced with those of $A(n)$, namely:
$$ \lambda (A^o(n)) = \{ 1,\ 3,\ 5,\ \ldots,\ 2n-3 \} .$$
A short proof verifies the claims.  In Section \S 3 we present some background theory concerning Jacobian matrices, and in Section \S4 apply our test matrix to a model of a physical spring-mass system, an application which leads naturally to Jacobian matrices.

\section{Main Result}

	Let $A(n)$ be a real, symmetric, tridiagonal, $n \times n$ matrix with entries
	\begin{align*}
		a_{ii} &= n-1 & i=1,2,\dots,n\\
		a_{i,i+1} &= \frac{1}{2} \sqrt{i (2n - i - 1)} & i = 1, 2, \dots, n-2 \\
		a_{n-1,n} &= \sqrt{\frac{n(n-1)}{2}}
	\end{align*}
	and let $\ds A^o(n)$ be the principal submatrix of $A(n)$, that is, the $(n-1) \times (n-1)$ matrix obtained from $A(n)$ by deleting the last row and column.
	
	\begin{thm}
		$A(n)$ has eigenvalues $\{0, 2, \dots, 2n-2 \}$ and $\ds A^o(n)$ has eigenvalues $\{ 1, 3, \dots, 2n -3 \}$.
	\end{thm}
	
	\begin{proof}
		By induction.  When $n=2$
			$$A(2) = \begin{bmatrix} 1&1 \\ 1&1 \end{bmatrix}$$
		has eigenvalues $0, 2$, and $\ds A^o(2)$ has eigenvalue $1$.  Assume the result holds for $n$.  So $A(n)$ has eigenvalues $\{ 0, 2, \dots, 2n - 2 \}$.  Let $\ds B = A^o(n+1) - nI$ and $A = A(n) - (n-1)I$.  Then $B$ and $A$ are similar via $BR=RA$ where $R$ is upper triangular, with entries
			$$r_{ij} = \left\{ \begin{array}{ll} \sqrt{\frac{k(j-1)!(2n-j-1)!}{(i-1)!(2n-i+1)!}} & i,j \hbox{ have same parity and } j \geq i, \\
			0 & \hbox{otherwise,} \end{array} \right.$$
		and
			$$k = \left\{ \begin{array}{ll} 2 & j \not= n, \\
			1 & j=n. \end{array} \right.$$
		Therefore $\ds A^o(n+1)$ has eigenvalues $\{ 1, 3, \dots, 2n-1 \}$.
		
		Now we show that $A(n+1)$ has eigenvalues $\{2n \} \cup \{ \hbox{eigenvalues of } A(n) \}$.  Let $C = A(n+1) - 2nI$.  Factorize $C = - LL^T$, where $L$ is lower bidiagonal.  We find:
		$$
\ba{ l@{\hspace*{0.8cm}}l@{\hspace*{0.8cm}}l}
			l_{ii} = \sqrt{\frac{2n-i+1}{2}} ; &
			l_{i+1,i} = - \sqrt{\frac{i}{2}}, &  i = 1, 2, \dots, n-1, \\
			l_{nn} = \sqrt{\frac{n+1}{2}}; &
			l_{n+1,n} = -\sqrt{n}; &
			l_{n+1,n+1} = 0.
		\ea
$$
		Therefore $C$ has eigenvalue $0$ and thus $A(n+1)$ has eigenvalue $2n$.
		
		Define $D = 2nI - L^TL$, so
			$$D = \begin{bmatrix} {\ds D^o }& O \\ O & 2n \end{bmatrix}$$
		with
$$
\ba{ l@{\hspace*{0.8cm}}l@{\hspace*{0.8cm}}l @{\mbox{ and }} l}
			d_{ii} = \frac{2n-1}{2}; & d_{i+1,i} = \frac{1}{2} \sqrt{i(2n-i)}, & i = 1, 2, \dots, n-1, & \ d_{nn} = \frac{n-1}{2}. \ea
$$
		Now $\ds D^o$ has the same eigenvalues as $A(n)$ since they are similar matrices via $\ds SD^o = A(n)S$ where $S$ is upper triangular with entries
		\begin{align*}
			s_{ii} &= \sqrt{2n-i}; \mi s_{i,i+1} = -\sqrt{i} \mhi \  i = 1, 2, \dots, n-1, \\
				s_{nn} &= \sqrt{2n}; \mi \mi s_{ij} = 0 \mi \mbox{otherwise}.
		\end{align*}
		Therefore $A(n+1)$ has eigenvalues $\{2n \} \cup \{ \hbox{eigenvalues of } A(n) \}$.
	\end{proof}

\section{Discussion}
A real, symmetric $n \times n$ tridiagonal matrix $B$ is called a {\em Jacobian}\/ matrix when its off-diagonal elements are non-zero (\cite{ipv}, p. 46).
We write
 \be{2}
B = \left[ \ba{rrrccc} 
    a_1 & -b_1 & 0 & 0 & \cdots & 0 \\ 
    -b_1 & a_2 & -b_2 & 0  & \cdots & 0 \\
     0 & -b_2 & a_3 & -b_3 &  \ddots & \vdots \\
     0 & 0 & \ddots & \ddots & \ddots & 0 \\
      \vdots & \vdots & \ddots  & -b_{n-2} & a_{n-1} & - b_{n-1} \\
      0 & 0 & \ldots & 0 & -b_{n-1} & a_n \ea \right]
 \ee
The similarity transformation, $\hat{B} = S^{-1}BS$, where $S=S^{-1}$ is the alternating sign matrix, $S = \mbox{diag} (1, -1, 1, -1, \ldots, (-1)^{n-1}),$ produces a Jacobian matrix $\hat{B}$ with entries the same as $B$ except for the sign of the off-diagonal elements, which are all reversed. If  instead we use the self-inverse sign matrix, $S^{(m)} = \mbox{diag}(\underbrace{1,1,\ldots, 1}_{m},\overbrace{-1,-1,\ldots, -1}^{n-m}),$ to transform $B$, then $\hat{B}$ is a Jacobian matrix  identical to $B$ except for a switched sign on the $m^{\mbox{\scriptsize th}}$ off-digonal element. In regards to the spectrum of the matrix, there is therefore no loss of generality in accepting the convention that a Jacobian matrix be expressed with negative off-diagonal elements, that is, $b_i>0, \mhi \forall \ i= 1, \ldots, n-1$ in \eqref{2}.

While Cauchy's interlace theorem \cite{hwang} guarantees that the eigenvalues of any square, real, symmetric (or even Hermitian) matrix will interlace those of its leading (or trailing) principal submatrix, the interlacing cannot be strict, in general \cite{fisk}. However, specializing to the case of Jacobian matrices restricts the interlacing to strict inequalities. That is, Jacobian matrices possess distinct eigenvalues, and the eigenvalues of the leading (or trailing) principal submatrix are also distinct, and strictly interlace those of the original matrix (see \cite{ipv}, Theorems 3.1.3 and 3.1.4. See also \cite{mc} exercise P8.4.1, p. 475: when a tridiagonal matrix has algebraically multiple eigenvalues, the matrix fails to be Jacobian). The inverse problem is also well-posed: there is a unique (up to the signs of the off-diagonal elements) Jacobian matrix $B$ having given spectra specified as per \eqref{1} (see \cite{ipv}, Theorem 4.2.1, noting that the interlaced spectrum of $n-1$ eigenvalues $\ds (\lambda^o)_1^{n-1}$ can be used to calculate the last components of each of the $n$ ortho-normalized eigenvectors of $B$ via equation 4.3.31). Therefore, the matrix $A(n)$ in theorem 2.1 is the unique Jacobian matrix with eigenvalues equally spaced by two, starting with smallest eigenvalue zero, whose leading principal submatrix has eigenvalues also equally spaced by two, starting with smallest eigenvalue one.

As a consequence of the theorem, we now have
\begin{cor}
The eigenvalues of the real, symmetric $n\times n$ tridiagonal matrix
\be{3}
W_n= \left[ \ba{cccccc} 
    a & -c\sqrt{\frac{n-1}{2}} & 0 & 0 & \cdots & 0 \\ 
     -c\sqrt{\frac{n-1}{2}}  & a &   -c\sqrt{\frac{2n-3}{2}} & 0 & \cdots & 0 \\
     0 & -c\sqrt{\frac{2n-3}{2}} & a &  -c\sqrt{\frac{3n-6}{2}}  & \ddots & \vdots \\
%    0 & 0 & -c\sqrt{\frac{3n-6}{2}} & a & -c\sqrt{\frac{4n-10}{2}} \ddots & 0 \\
     0 & 0 & \phantom{\sqrt{\frac{n(n-1)}{2}}}\hspace*{-15mm} \ddots & \ddots & \ddots & 0 \\
%    \vdots & \vdots & \ddots &  -c\sqrt{\frac{(n-3)(n+2)}{4}}  & a & -c\sqrt{\frac{(n-2)(n+1)}{4}} & 0 \\
      \vdots  & \vdots & \ddots  & -c\sqrt{\frac{(n-2)(n+1)}{4}}  & a & -c\sqrt{\frac{n(n-1)}{2}} \\
      0 & 0 & \ldots &  0 & -c\sqrt{\frac{n(n-1)}{2}} & a \ea \right]
\ee
 form the arithmetic sequence, 
\be{4} \lambda (W_n) = \{ a_o + 2c (i-1) \}_{i=1}^{n}, \ee 
while the eigenvalues of its leading principal submatrix, $\ds W_n^{o}$, form the uniformally interlaced sequence 
\be{5}  \lambda (W_n^o ) = \{ a_o +c + 2c (i-1) \}_{i=1}^{n-1}, \mhi \mbox{where } \ a = a_o + c(n-1).  \ee
\end{cor}

The form and properties of $W_n$ were first hypothesised by the third author while programming Fortran algorithms to reconstruct band matrices from spectral data \cite{nbw}. Initial attempts to prove the spectral properties of $W_n$ by both he and his graduate supervisor (the first author) failed. Later, the first author produced the short induction argument of Theorem (2.1), in July, 1996. Alas, the fax on which the argument was communicated to the third author was lost in a cross-border academic move, and so the matter languished until recently. In summer of 2013, the second and third authors, looking for a tractable problem for a summer undergraduate research project, assigned the problem of this paper: ``hypothesize, and then verify, if possible, the explicit entries of an $n\times n$ symmetric, tridiagonal matrix with eigenvalues \eqref{4}, such that the eigenvalues of its principal submatirx are \eqref{5}.'' Meanwhile a building renovation in summer of 2013 required the third author to clean out his office, including boxes of old papers and notes. During this process the misplaced fax from the first author was found.  However, to our delight, the student, Mr.\ Alexander De Serre-Rothney, was also able to independently complete both parts of the problem. His proof is now found in \cite{asr}. Though longer than the one presented here, his proof utilizes the spectral properties of another tridiagonal (non-symmetric) matrix, the so-called {\em Kac-Sylvester matrix},\/ $K_n$,  of size $(n+1) \times (n+1)\ $ \cite{ cle, edelman, muir, tt}:
 \be{6}
K_n = \left[ \ba{rrrccc} 
    n & n-1 & 0 & 0 & \cdots & 0 \\ 
    1 & n & n-2 & 0  & \cdots & 0 \\
     0 & 2 & n & n-3 &  \ddots & \vdots \\
     0 & 0 & \ddots & \ddots & \ddots & 0 \\
      \vdots & \vdots & \ddots  & n-2 & n & 1 \\
      0 & 0 & \ldots & 0 & n-1 & n \ea \right] \ ,
 \ee
which has eigenvalues $\lambda (K_n) = \{ 2k-n \}_{k=0}^n$.

\section{A Spring-Mass Model problem}
		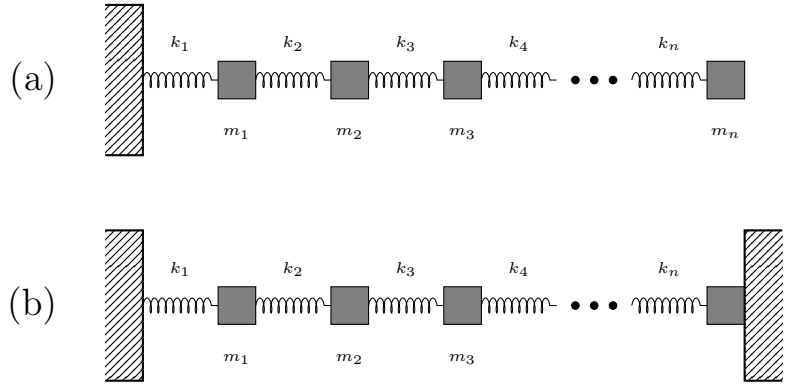
\begin{figure}[!h]
		\begin{center}
		\begin{tikzpicture}[font=\tiny]
		\draw (-1,3) node[left] {\large (a)};
		\draw (-1,0) node[left] {\large (b)};
		\foreach \i in {0,3}
			{
			\fill[pattern = north east lines] (-0.5,1+\i) rectangle (0,-1+\i);
			\draw[thick,black] (-0.5,1+\i) -- (0,1+\i) -- (0,-1+\i) -- (-0.5,-1+\i);
			\draw[decoration={aspect=0.3, segment length=1mm, amplitude=1mm,coil},decorate] (0,0+\i) -- (1,0+\i);
			\fill[fill=gray,draw=black] (1,0.25+\i) rectangle (1.5,-0.25+\i);
			\draw[decoration={aspect=0.3, segment length=1mm, amplitude=1mm,coil},decorate] (1.5,0+\i) -- (2.5,0+\i);
			\fill[fill=gray,draw=black] (2.5,0.25+\i) rectangle (3,-0.25+\i);
			\draw[decoration={aspect=0.3, segment length=1mm, amplitude=1mm,coil},decorate] (3,0+\i) -- (4,0+\i);
			\fill[fill=gray,draw=black] (4,0.25+\i) rectangle (4.5,-0.25+\i);
			\draw[decoration={aspect=0.3, segment length=1mm, amplitude=1mm,coil},decorate] (4.5,0+\i) -- (5.5,0+\i);
			\draw[fill=black] (5.75,0+\i) circle (0.5mm);
			\draw[fill=black] (6,0+\i) circle (0.5mm);
			\draw[fill=black] (6.25,0+\i) circle (0.5mm);
			\draw[decoration={aspect=0.3, segment length=1mm, amplitude=1mm,coil},decorate] (6.5,0+\i) -- (7.5,0+\i);
			\fill[fill=gray,draw=black] (7.5,0.25+\i) rectangle (8,-0.25+\i);
			
			\draw (0.5,0.25+\i) node[above] {$k_1$};
			\draw (2,0.25+\i) node[above] {$k_2$};
			\draw (3.5,0.25+\i) node[above] {$k_3$};
			\draw (5,0.25+\i) node[above] {$k_4$};
			\draw (7,0.25+\i) node[above] {$k_n$};
			\draw (1.25,-0.5+\i) node[below] {$m_1$};
			\draw (2.75,-0.5+\i) node[below] {$m_2$};
			\draw (4.25,-0.5+\i) node[below] {$m_3$};
			}
			\draw (7.75,2.5) node[below] {$m_n$};
%			\draw[decoration={aspect=0.3, segment length=1mm, amplitude=1mm,coil},decorate] (8,0) -- (9,0);
			\fill[pattern = north east lines] (8,1) rectangle (8.5,-1);
			\draw[thick,black] (8.5,1) -- (8,1) -- (8,-1) -- (8.5,-1);
%			\draw (8.5,0.25) node[above] {$k_{n+1}$};
		\end{tikzpicture}
		\end{center}
		\caption{Spring-mass system: (a) right hand end free, (b) right hand end fixed} \label{fig:spring_mass}
		\end{figure}
		
		Let 
		$$C = \begin{bmatrix} k_1+k_2 & -k_2 \\ -k_2 & k_2+k_3 & -k_3 \\  & \cdot & \cdot & \cdot & \\ && -k_{n-1} & k_{n-1} + k_n & - k_n \\ && & -k_n & k_n
		\end{bmatrix} \hspace*{5mm} \hbox{ and } \hspace*{5mm} M = \begin{bmatrix} m_1 \\ &m_2 \\ &&\cdot\\ &&&m_{n-1}\\&&&&m_n \end{bmatrix}$$
		
		Then the natural frequencies of the systems in Figure \ref{fig:spring_mass} are solutions of $(C-\lambda M) \vec{x} = \vec{0}$ and $(C^o - \lambda^o M^o)\vec{x}^o = \vec{0}$, where $C^o$ is obtained from $C$ by deleting the last row and column.  The solutions can be ordered $0 < \lambda_1 < \lambda^0_1 < \lambda_2 < \cdots < \lambda_{n-1} < \lambda^o_{n-1} < \lambda_n$.  We can also rewrite the systems as $(B - \lambda I) \vec{u} = \vec{0}$ and $(B^o-\lambda^o I)\vec{u}^o = \vec{0}$ where $B = M^{-1/2}CM^{-1/2}$ and $\vec{u} = M^{1/2}\vec{x}$.  Note the natural frequencies of the systems are the eigenvalues of $B$ and $B^o$.
		
		We will consider the system with natural frequencies $\{1, 3, \dots, 2n-1 \}$ for system (a), and $\{2, 4, \dots, 2n-2\}$ for system (b).  We can use the matrix $A(n)$ to help solve this system for the spring stiffnesses and the masses.
		
		Let $B(n)$ be the symmetric tridiagonal matrix:
			\begin{align*}
				B_{ii} &= a_i = n & i=1,\dots,n \\
				B_{i,i+1} &= -b_i = -\frac{1}{2} \sqrt{i(2n-i-1)} & i=1, \dots n-2 \\
				B_{n-1,n} &= -b_{n-1} = -\sqrt{\ds \frac{n(n-1)}{2}}
			\end{align*}
		Note that $B(n) = A(n) + I$ and has eigenvalues $\{ 2k + 1 \}_{k=0}^{n-1}$, while $B^o(n)$ has eigenvalues $\{ 2k \}_{k=1}^{n-1}$.
		
		Let $\vec{u} = \langle m_1^{1/2}, \cdots, m_n^{1/2} \rangle^T$ with $m_i > 0$ for all $i$.  Let $m = \sum_{i=1}^n m_i = \vec{u}^T \vec{u}$.  We wish to solve 
			\begin{align}
				B(n)\vec{u} &= \langle m_1^{-1/2} k_i, 0, \dots, 0 \rangle^T \label{eq:Bu}
			\end{align}
		for $(m_i)_{i=1}^n$, and $k_1$.
		
		The bottom, $n^{th}$, equation is 
			$$m_{n-1}^{1/2} =\ds \frac{-n m_n^{1/2}}{-b_{n-1}} = \sqrt{2} \left( \frac{n}{n-1} \right)^{1/2} \alpha,$$
		where we choose $m_n^{1/2} = \alpha$.  We will thus be able to express $m_i^{1/2}$ in terms of the scaling parameter $\alpha$.
		
		The $(n-1)^{th}$ equation is
			$$m_{n-2}^{1/2} = \frac{\alpha b_{n-1} - n m_n^{1/2}}{-b_{n-2}}=\alpha \frac{\ds \sqrt{\frac{n(n-1)}{2}} - n \sqrt{2} \sqrt{\frac{n}{n-1}}}{-\frac{1}{2}\sqrt{(n-2)(n+1)}} = \alpha \sqrt{2} \left( \frac{n(n+1)}{(n-1)(n-2)} \right)^{1/2}.$$
			
		The $i^{th}$ equation, for $i \not= 1, n-1, n$, is
			$$-b_{i-1} m_{m-i}^{1/2} + n m_i^{1/2} - b_i m_{i+1}^{1/2} = 0.$$
		Then
			\begin{align}
				m_{i-1}^{1/2} &= \frac{2n m_i^{1/2} - (i(2n-i-1))^{1/2} m_{i+1}^{1/2}}{((i-1)(2n-i))^{1/2}} \label{eq:m_i}
			\end{align}
			
		Now suppose
			\begin{align}
				m_{n-i}^{1/2} = \alpha \sqrt{2} \left( \frac{n(n+1)\cdots(n+i-1)}{(n-1)(n-2) \cdots (n-i)} \right)^{1/2} \label{eq:m_n-i}
			\end{align}
		for $i = 1, 2, \dots, j$.  Then cases $i=1,2$ are already verified, and the strong inductive assumption applied in \eqref{eq:m_i} with $i-1=n-(j+1)$ implies $i=n-j$.  So
			\begin{align*}
				m_{n-j-1} &= \frac{2n \alpha \sqrt{2} \left( \frac{n(n+1)\cdots(n+j-1)}{(n-1)(n-2) \cdots (n-j)} \right)^{1/2}  - ((n-j)(n+j-1))^{1/2} m_{n-j+1}^{1/2}}{((n-j-1)(n+j))^{1/2}} \\
				&= \alpha \sqrt{2} \left( \frac{n(n+1)\cdots(n+j-1)}{(n-1)(n-2) \cdots (n-j)} \right)^{1/2} \left[ \frac{2n\left(\frac{n+j-1}{n-j}\right)^{1/2} - ((n-j)(n+j-1))^{1/2}}{((n-j-1)(n+j))^{1/2}} \right] \\
				&= \alpha \sqrt{2} \left( \frac{n(n+1)\cdots(n+j-1)}{(n-1)(n-2) \cdots (n-j)} \right)^{1/2} \left[ \frac{2n - (n-j)}{((n-j-1)(n+j))^{1/2}} \right] \\
				&= \alpha \sqrt{2} \left( \frac{n(n+1)\cdots(n+j-1)(n+j)}{(n-1)(n-2) \cdots (n-j)(n-j-1} \right)^{1/2}
			\end{align*}
		which verifies, by strong induction, the closed form for $m_{n-i}^{1/2}$ given by \eqref{eq:m_n-i}.
		
		Finally, the first equation of \eqref{eq:Bu} is
			$$nm_1^{1/2} - b_1 m_2^{1/2} = m_1^{-1/2}k_1$$
		and so
			$$k_1 = n m_1 - b_1(m_1 m_2)^{1/2}.$$
			
		We note that the values $m_{n-i}$ can be written as
			\be{7} 
m_{n-i} = 2\alpha^2 \frac{(n+i-1)!(n-i-1)!}{((n-1)!)^2}
\ee
		for $i=1, \dots, n-1$, and $m_n = \alpha^2 \ds \frac{(n+0-1)!(n-0-1)!}{((n-1)!)^2} = \alpha^2$.
		
		Since $C = M^{1/2}B(n)M^{1/2}$, then
	\be{8}
k_{i+1} = -C_{i,i+1} = -m_{n-(n-i)}^{1/2}B_{i,i+1} m_{n-(n-i-1)}^{1/2} = \alpha^2 \left( \frac{i!(2n-i-1)!}{((n-1)!)^2} \right)
\ee
		and $k_1 = \ds \alpha^2 \frac{(2n-1)!}{((n-1)!)^2}.$
\section{Conclusion}
A family of $n\times n$ symmetric tri-diagonal matrices, $W_n$,  whose eigenvalues are simple and uniformally spaced, and whose leading principle submatrix has uniformly interlaced, simple eigenvalues, has been presented \eqref{3}. Members of the family are characterized by a specified smallest eigenvalue, $a_o$ and gap size, $c$ between eigenvalues.
The matrices are termed  {\em Jacobian,}\/ since the off-diagonal entries are all non-zero.  The matrix entries are explicit functions of the size $n$, $a_o$ and $c$, so the matrices can be used as a test matrices for eigenproblems, both forward and inverse.  The matrix $W_n$ for specified smallest eigenvalue $a_o$ and gap $c$ is unique up to the signs of the off-diagonal elements. 

In Section \S 4, the form of $W_n$ was used as an explicit solution of a spring-mass vibration model (Figure 4.1), and the inverse problem to determine the lumped masses and spring stiffnesses was solved explicitely. Both the lumped masses, $m_{n-i}$ given by equation \eqref{7}, and spring stiffnesses, $k_{n-i}$ from equation \eqref{8} show super-exponential growth. Consequently $m_n/m_1, k_n/k_1$ become vanishingly small as $n\rightarrow \infty$. As a result, the spring-mass system of Figure 4.1 cannot be used as a discretized model for a physical beam in longitudinal vibration, as the model becomes unrealistic in the limit as $n \rightarrow \infty$. 

\bibliography{gladwell}
\bibliographystyle{plain}

\end{document}